\documentclass[11pt]{amsart}
\theoremstyle{plain}
\newtheorem{Thm}{Theorem}

\newtheorem{Pro}[Thm]{Proposition}
\newtheorem{Lem}[Thm]{Lemma}

\begin{document}

\title[New Yamabe-type flow in a manifold]
{New Yamabe-type flow in a compact Riemannian manifold}

\author{Li Ma$^*$}

\address{Li Ma, School of Mathematics and Physics \\
University of Science  and technology Beijing\\
Xueyuan Road 30, Haidian\\
Beijing 100083 \\
China} \email{lma17@ustb.edu.cn}

\address{ Department of Mathematics \\
Henan Normal University \\
Xinxiang, 453007 \\
China}

 \dedicatory{}
\date{Feb. 6th, 2020}


\thanks{$^*$ The research of Li Ma is partially supported by the National Natural Science
Foundation of China No. 11771124 and a research grant from USTB, China.}

\begin{abstract}
In this paper, we set up a new Yamabe type flow on
a compact Riemannian manifold  $(M,g)$ of dimension $n\geq 3$. Let $\psi(x)$ be any smooth function on $M$. Let $p=\frac{n+2}{n-2}$ and $c_n=\frac{4(n-1)}{n-2}$.
We study the Yamabe-type flow $u=u(t)$ satisfying
$$
{u_t}=u^{1-p}(c_n\Delta u-\psi(x)u)+r(t)u, \ \ in \ M\times (0,T),\  T>0
$$
with
$$
r(t)=\int_M(c_n|\nabla u|^2+\psi(x)u^2)dv/ \int_Mu^{p+1},
$$
which preserves the $L^{p+1}(M)$-norm and we can show that for any initial metric $u_0>0$, the flow exists globally.
  We also show that in some cases, the global solution converges to a smooth solution to the equation
  $$
  c_n\Delta u-\psi(x)u+r(\infty)u^{p}=0, \ \ on \ M
  $$
  and our result may be considered as a generalization of the result of T.Aubin, Proposition in p.131 in \cite{A82}.
\end{abstract}

 \maketitle

\section{Introduction}
  Nonlocal evolution equations arise naturally from geometry. The most famous one is the normalized Ricci flow preserving the volume \cite{H82}, introduced by R.Hamilton in 1982. The evolutions of planar curves preserving the length or area enclosed \cite{GA} \cite{DM} \cite{MC1} are in this category.  To solve the Yamabe problem from the view point of evolution equation, Hamilton has also proposed the normalized Yamabe flow to approaching a Yamabe metric on a closed manifold.
In this paper, we introduce a Yamabe type flow ( which preserves the $L^{\frac{2n}{n-2}}(M)$-norm, see below for the definition) on a compact Riemannian manifold and study its global existence and convergence in some cases. We point out that some arguments in \cite{Y} and \cite{B05} about Yamabe flow can be used to handle such a general norm-preserving flow.  We also notice that many heat flow methods may be introduced to functionals related to Yamabe problem on $(M, g)$ (\cite{A82} \cite{B97}, \cite{B}, \cite{MC}).

We now introduce a new Yamabe-type flow on
a compact Riemannian manifold  $(M,g)$ of dimension $n\geq 3$. Let $p=\frac{n+2}{n-2}$, $c_n=\frac{4(n-1)}{n-2}$, and let $R(g)$ be the scalar curvature. Assume that $\psi$ is a given smooth function in $M$.
The Yamabe-type flow $u=u(t)$ is defined such that $u(t)$ satisfies the evolution equation
\begin{equation}\label{yamabe}
{u_t}=u^{1-p}(c_n\Delta u-\psi(x)u)+r(t)u, \ \ in \ M\times (0,T),\  T>0
\end{equation}
where
$$
r(t):=r_\psi(t):=\int_M(c_n|\nabla u|^2+\psi(x)u^2)dv/ \int_Mu^{p+1}
$$
with initial data $u(0)>0$. The local in time solution of this problem \eqref{yamabe} is by now standard \cite{M} and can be obtained by the fixed point method or the method such as the implicit function theorem. This flow preserves the norm of the evolving function $u(t)$,
$$
\int_M u(t)^{p+1}dv=\int_M u(0)^{p+1}dv
$$
which may be assumed to be one for simplicity. In fact, we have
\begin{align*}
\frac{1}{p+1}\frac{d}{dt}\int_M u^{p+1}dv&=\int_M u^pu_t   \\
&=\int_M u(c_n\Delta u-\psi(x)u)+ r(t)\int_Mu^{p+1}\\
&=0.
\end{align*}
Hence,
$$
\int_M u(t)^{p+1}dv=\int_M u(0)^{p+1}dv=1.
$$
Since, for $u=u(t)$,
$$
r(t)=\int_M(c_n|\nabla u|^2+\psi(x)u^2)dv, \ \ \int_M u(t)^{p+1}=1,
$$
we have
\begin{align*}
\frac12\frac{d}{dt}r(t)&=\int_M c_n(\nabla u,\nabla u_t)+\psi(x) uu_t   \\
&=\int_M (-c_n\Delta u+\psi(x)u)u_t\\
&=\int_M u^p(-\frac{u_t}{u}+r(t))u_t\\
&=-\int_M u^p\frac{u^2_t}{u}\leq 0.
\end{align*}
So, $r(t)$ is non-increasing in $t$ along the flow.

To understand this flow well, we introduce the pseudo-scalar curvature
\begin{equation}\label{scalar}
R_\psi=u^{-p}(-c_n\Delta u+\psi(x)u).
\end{equation}
Then the equation \eqref{yamabe} can be written as
\begin{equation}\label{yamabe2}
 u_t =u(-R_\psi+r(t)), \ \ in \ M\times (0,T),\  T>0.
\end{equation}
We remark that one may study the flow
\begin{equation}\label{yamabe-open}u_t =-R_\psi u
\end{equation}
 on any complete non-compact Riemannian manifold of dimension $n\geq 0$.

For the evolution problem \eqref{yamabe} on the compact Riemannian manifold $(M,g)$ of dimension $n\geq 0$, we can show that for any initial data $u(0)>0$, the flow exists globally.

\begin{Thm}\label{thm1} For any initial data $u(0)>0$, the Yamabe-type flow $(u(t))$ to \eqref{yamabe} above exists globally and
$$r_\infty=\lim_{t\to \infty} r(t)=r(\infty)$$ exists.
\end{Thm}

One may give a proof of this result using similar arguments as in section 4 in \cite{Y} or as in \cite{B05}, which is a local in natural argument for the solution $u$ to Yamabe flow. Here, we prefer to give a direct proof to control the norm growth of pseudo-scalar curvature along the flow. In the case when $(M,g)$ is a closed surface, we may also introduce the $\psi$-Gauss flow flow. Let $g(t)=e^{2u(t)}g$, where $u(t):M\to \mathbb{R}$ is a smooth function, and let $\psi:M\to\mathbb{R}$ be a given smooth function. The $\psi$-Gauss flow is defined by 
\begin{equation}\label{gauss}
e^{2u}u_t=\Delta u-\psi(x)+r(t)e^{2u}, \ \ in \ M\times (0,T),\  T>0
\end{equation}
where
$$
r(t):=r_\psi(t):=\int_MK_\psi dv/ \int_Me^{2u}
$$
where $K_\psi =e^{-2u}(-\Delta u+\psi(x))$ with initial data $u(0)>0$. By similar method, we know that there is a global flow for \eqref{gauss}. Interesting questions are to find similar results to Chang-Yang \cite{CY}, \cite{CD}, and Ding-Liu \cite{DL}.
Related Yamabe type flow with boundary data may also be studied.

We can also get the convergence result of the flow $(g(t))$ to \eqref{yamabe} as in the Yamabe-scalar negative and zero cases. So we may define the Yamabe-type invariant below.
Define, for $p=\frac{n+2}{n-2}$, for $u\in H^1(M); u\not=0$,
$$
E(u)=\frac{\int_M(c_n|\nabla u|^2+\psi u^2)dv}{(\int_M |u|^{p+1}dv)^{2/(p+1)}}
$$
and
$$
Y_\psi(M)=\inf\{ E(u), u\in H^1(M); u\not=0\}
$$
which is called the Yamabe-type invariant of $M$. We denote for $M=S^n$ and $\psi=n(n-1)$
$$
Y(S^n)=\inf\{ E(u);u\in H^1(M); u\not=0\}
$$
for the Yamabe constant on $S^n$. Using Aubin's argument (see p.131 in \cite{A82},see also \cite{S} and \cite{E}), we know that $Y_\psi(M)\leq Y(S^n)$. In \cite{A76}, Aubin proved that if $n\geq 4$ and $\psi(x)<R(g)(x)$ somewhere, then there is a minimizer for $Y_psi(M)$. Of course, one may use the argument of Brezis-Nirenberg \cite{BN83} to know
 $Y_\psi(M)<Y(S^n)$ provided $\psi<n(n-1)$ for $M^n=S^n$ with $n\geq 4$.

We define
$$
\lambda_1(\psi)=\inf_{\{u\in H^1(M); u\not=0\}} \frac{\int_M(c_n|\nabla u|^2+\psi u^2)dv}{\int_M |u|^2dv}.
$$
Then it is standard to know that there is a positive function $u_1$ such that
$$
\lambda_1(\psi)=\int_M(c_n|\nabla u_1|^2+\psi u_1^2)dv, \ \ \int_M u_1^{2}dv=1,
$$
and
$$
-c_n\Delta u_1+\psi(x)u_1=\lambda_1(\psi)u_1,\  \ in \ M.
$$
We remark that for $\lambda_1(\psi)\geq 0$, we have $Y_\psi(M)\geq 0$. In fact, for ant $u\not=0$,
\begin{align*}
\frac{\int_M(c_n|\nabla u|^2+\psi u^2)dv}{(\int_M |u|^{p+1}dv)^{2/(p+1)}}&=\frac{\int_M(c_n|\nabla u|^2+\psi u^2)dv}{\int_M u^{2}dv} \frac{\int_M u^{2}dv}{(\int_M |u|^{p+1}dv)^{2/(p+1)}} \\
&\geq \lambda_1(M)\frac{\int_M u_1^{2}dv}{(\int_M |u|^{p+1}dv)^{2/(p+1)}} \geq 0.
\end{align*}
 For $\lambda_1(\psi)< 0$, we have $Y_\psi(M)< 0$. In fact, taking $u=u_1$ above, we have
\begin{align*}
 Y_\psi(M)&\leq \frac{\int_M(c_n|\nabla u|^2+\psi u^2)dv}{\int_M u^{2}dv} \frac{\int_M u^{2}dv}{(\int_M |u|^{p+1}dv)^{2/(p+1)}}\\
  &=\lambda_1(M)\frac{\int_M u^{2}dv}{(\int_M |u|^{p+1}dv)^{2/(p+1)}}<0.
 \end{align*}

The relation between $\lambda_1(M)$ and $Y_\psi(M)$ can be given below.
Since, by the Holder inequality, we have
$$
\int_M |u|^2dv\leq (Vol(M))^{\frac{p-1}{p+1}}(\int_M |u|^{p+1}dv)^{2/(p+1)}.
$$
We then have
$$
Y_\psi(M)\leq (Vol(M))^{-\frac{p-1}{p+1}}\lambda_1(\psi), \
$$
for $\lambda_1(\psi)\geq 0$ and
$$
Y_\psi(M)\geq (Vol(M))^{-\frac{p-1}{p+1}}\lambda_1(\psi),
$$
for $\lambda_1(\psi)< 0$. Note that if $\psi(x)\leq 0$ on $M$, then $\lambda_(M)<0$. In fact, we may take $u=1/\sqrt{vol(M)}$. Then,
$\lambda_1(M)\leq \frac{1}{col(M)}\int_M\psi dv<0$.

Our main result is below.
\begin{Thm}\label{main}
Assume $0<Y_\psi(M)<Y(S^n)$ and assume, for the initial metric $g_0= u_0^{4/(n-2)} g$ with $u_0>0$ on $M$, $E(u_0)\leq Y(S^n)$. Then along the Yamabe-type flow $(u(t))$ to \eqref{yamabe}, we have a convergent subsequence $u(t_j)\to u_\infty>0$, $t_j\to\infty$, and $u_\infty$ is a smooth function satisfies
$$
-c_n\Delta u+\psi(x)u=r_\infty u^{p}, \ \ in \ M, \ \int_M u^{p+1}=1.
$$
\end{Thm}
Our result may be considered as a generalization of Proposition in p.131 in \cite{A82}.
We remark that with more detailed analysis (see \cite{BB}), one may obtain similar result to Theorem 1.1 in \cite{B05} and we leave this open for interested readers for pleasure.
Whether the Yamabe-type invariant on $(M,g)$ can be achieved by some smooth function $u>0$ in $M$, generally speaking, is still an open problem and may be discussed in latter chances.

Assume that $\lambda_1(\psi)<0$ and $\psi(x)<0$ on $M$,
we can show that the flow converges at time infinity.

\begin{Thm}\label{thm2}
  Assume that $\lambda_1(\psi)< 0$ and $\psi(x)<0$ on $M$. The Yamabe-type flow $(u(t))$ converges to a metric of constant pseudo-scalar curvature at $t=\infty$.
\end{Thm}

We remark that the results may be extended to the case when $p>1$. In fact, assuming that $\lambda_1(\psi)=0$ and $\psi(x)=0$ on $M$,  we have the following result.

\begin{Thm}\label{thm3}
  Assume that $\lambda_1(\psi)= 0$ and $\psi(x)=0$ on $M$. Fix any $p>1$. The Yamabe-type flow $(u(t))$  satisfying
  \begin{equation}\label{yamabe5}
\frac{u_t}{u}=u^{-p}c_n\Delta u+r(t), \ in \ M\times (0, T), \ T>0
\end{equation}
 with
 $$
 r(t)=\int_Mc_n|\nabla u|^2dv/ \int_Mu^{p+1}
 $$
  exists globally and
  converges to a positive constant at $t=\infty$.
\end{Thm}

Via a use of bubble analysis, we can handle more complicated case as in \cite{B05}, since the proof is lengthy, we prefer to present it elsewhere.

The plan of this note is below. The proof of Theorem \ref{thm1} will be given in section \ref{sect2} below. Using Struwe's compactness result, we prove Theorem \ref{main} in section \ref{sect3}. The proofs of Theorem \ref{thm2} and  Theorem \ref{thm3} will be given in section \ref{sect4}.

\section{global existence  of Yamabe-type flows }\label{sect2}

We treat the difficulty case when $\psi(x)>0$ and $r(0)>0$ and the other case can be handled in the proof of Theorem \ref{thm2} below.
We first establish the following result. Note that $Y_\psi(M)>0$ when $r>0$.

In short we denote by $R=R_\psi$. Recall $$r-R=\frac{u_t}{u}.$$
By \eqref{scalar} we know that
$$
-R=u^{-p}(c_n\Delta u-\psi(x)u).
$$
Taking the time derivative on both sides, we have
\begin{align*}
\frac{d}{dt}(-R)&=-pu^{-p-1}u_t(c_n\Delta u-\psi(x)u)+ u^{-p}(c_n\Delta u_t-\psi(x)u_t)  \\
&=p\frac{u_t}{u}R+u^{-p}(c_n\Delta [(r-R)u]-\psi(x)[(r-R)u])  \\
&=p(r-R)R+u^{-p}(c_n\Delta [(r-R)u]-\psi(x)[(r-R)u]),
\end{align*}
which can be written as
$$
\frac{\partial}{\partial t}R=u^{-p}(c_n\Delta [(R-r)u]-\psi(x)[(R-r)u])+p(R-r)R.
$$
Define
$$
L^uv=u^{-p}(c_n\Delta [vu]-\psi(x)[vu]),
$$
and then we have
$$
\frac{\partial}{\partial t}R=-L^u(r-R)+p(R-r)R.
$$

Then we may use the maximum principle to obtain
\begin{Pro} \label{lm1}
We have the pseudo-scalar curvature lower bound
\begin{equation}
\inf _{M} R_\psi(t) \geq \min \left\{\inf _{M} R_\psi(0), 0\right\}
\end{equation}
along the Yamabe-type flow.
\end{Pro}

\begin{proof} Note that at the minimum point of $R(t)$, we have
Recall that
$$
\frac{\partial}{\partial t} R\geq u^{-p}((R-r)c_n\Delta u-\psi(x)[(R-r)u])+p(R-r)R.
$$
Recall that
$$
R= u^{-p}(-\Delta u+\psi(x)u).
$$
Then we have
$$
\frac{\partial}{\partial t} R\geq (R-r)(-R)+p(R-r)R=(p-1)R(R-r).
$$
We may write it as
$$
\frac{\partial}{\partial t}(R\exp((p-1)\int_0^t (r-R)))\geq 0,
$$
which implies that
$$
R\exp((p-1)\int_0^t (r-R))\geq \inf_M R(0).
$$

Then we have the conclusion.
\end{proof}

Similarly, at the maximum point of $R$, we have
$$
\frac{\partial}{\partial t} R\leq (R-r)(-R)+p(R-r)R=(p-1)R(R-r),
$$
which implies that
$$
R\exp((p-1)\int_0^t (r-R))\leq \sup_M R(0),
$$
which is useful in the case when $\sup_MR(0)\leq 0$.

To get better estimate, we now choose $\sigma\geq 1$ such that
\begin{equation}
\sigma=\max \left\{\sup _{M}\left(1-R(0)\right), 1\right\}.
\end{equation}
Then applying the maximum principle again, we have
\begin{equation}
R_\psi(t)+\sigma \geq 1 \text { for all } t \geq 0
\end{equation}

At any finite time interval, we have Harnack inequality for the flow in the sense below.
\begin{Lem}\label{lm4}
\begin{equation}
\sup_{M} u(t) \leq C(T)
\end{equation}
and
\begin{equation}
 \inf_{M} u(t) \geq c(T)>0.
\end{equation}
\end{Lem}

\begin{proof} Note that

\begin{equation}
u^{-1}\frac{\partial}{\partial t} u(t)=-\left(R_\psi(t)-r(t)\right) \leq \left(r(0)+\sigma\right)
\end{equation}
 Thus,
\begin{equation}
\sup _{M} u(t) \leq C(T)
\end{equation}
for $t\in [0,T]$. Hence, for
\begin{equation}
P:=\psi(x)+\sigma\left(\sup _{0 \leq t \leq T \atop M} u(t)\right)^{p-1},
\end{equation}
we have
\begin{equation}
\begin{aligned} &-c_n\Delta u(t)+P u(t) \\ & \geq- c_n\Delta u(t)+\psi(x) u(t)+\sigma u(t)^{\frac{n+2}{n-2}} \\ &=\left(R_\psi(t)+\sigma\right) u(t)^{p} \geq 0 \end{aligned}
\end{equation}
for all $t\in [0,T]$.
Then using a Moser iteration argument (or by Cor. A.5 in \cite{B05})  we have
$$
\int_M u(t)\leq C(T)\inf_Mu(t).
$$
Then by the volume constrain condition, we have
\begin{equation}
1\leq C(T)\inf _{M} u(t)\left(\sup _{M} u(t)\right)^{p} .
\end{equation}
for all $t\in [0,T]$. Since $\sup_M u(t)\leq C(T)$, we get the conclusion.

\end{proof}

By now it is standard to set the global existence of the flow.
Using the result from \cite{T2} (see also \cite{L}), we have the result below.
\begin{Pro} \label{lm6}
For $T>0$, there exist $0<\alpha<1$ and a constant $C(T)$ such that
\begin{equation}
\left|u\left(x_{1}, t_{1}\right)-u\left(x_{2}, t_{2}\right)\right| \leq C(T)\left(\left(t_{1}-t_{2}\right)^{\frac{\alpha}{2}}+d\left(x_{1}, x_{2}\right)^{\alpha}\right)
\end{equation}
for all $x_1,x_2\in M$ and any $t_1,t_2\in [0,T]$ with $0<t_2-t_1<1$.
\end{Pro}

\begin{proof} Set $v=u^p$. The Yamabe type flow equation (\ref{yamabe}) on a compact sub-domain $D\subset M$ can be written as the divergence form that
$$
\partial_t v=c_ndiv (v^{1/p-1}\nabla v)- p\psi(x) v^{1/p-1} v +pr(t)v.
$$
One can see that the structure conditions (2.1) and (1.2-1.3) in \cite{T2} are satisfied.
Then we can invoke Theorem 4.2 in \cite{T2} to get the locally uniformly Holder estimate in $D$ for solutions $u$ up to the initial time  $t=0$. Namely, for any $x_0\in M$, $r>0$, and $T>0$, there exists uniform positive constants
$$\beta=\beta(B_{r}(x_0),sup_{B_{r}(x_0)}u(0))\in (0,1),$$ and $C(B_{r}(x_0),\beta,u(0), T)$ such that for any $j\geq 1$ with $B_{r}(x_0)\subset  D$, there holds
$$
|u|_{C^{\beta; \beta/2}(B_{r/2}(x_0)\times [0,T_0])}\leq C(B_{r}(x_0),sup_{B_{r}(x_0)}u(0), T).
$$
As always, we have used the Holder spaces $C^{\alpha; \alpha/2}(B_{r/2}(x_0)\times [0,T])$ in parabolic distance defined by $g+dt^2$.
Using the covering argument we can extend the estimate above to whole parabolic region $M\times [0,T]$.
\end{proof}

We now prove Theorem \ref{thm1}, which is the global existence result of the flow for any initial data.
\begin{proof} Fix any $T>0$. With the understanding of Proposition \ref{lm6},
we may use the standard regularity theory for parabolic equations (see \cite{F}, Theorem 5 on p. 64 or the book \cite{L}) to conclude that all higher order derivatives of the solution $u$ are uniformly bounded on every fixed time interval $[0,T]$. Hence,
we can extend the flow beyond $T$ and then the flow exists for all time.
This then completes the proof of Theorem \ref{thm1}.

\end{proof}

\section{P.S.sequence of Yamabe-type flows and proof of Theorem \ref{main}}\label{sect3}

To understand the asymptotic behavior of the Yamabe type flow \eqref{yamabe}, we need some integral estimates about the scalar curvature $R$. Then we consider the P.S. Sequence of the functional $E(u)$. Note that once we have the P.S.sequence along the flow, we may invoke the by now standard argument, that is, Struwe's compactness result \cite{St} to get the partial compactness of the sequence.

Let $\breve{R}=R-r$. Then
$$
\frac12 \frac{d}{dt} \breve{R}^2=\breve{R}\breve{R}_t=\breve{R}L^u(\breve{R})+pR\breve{R}^2-\breve{R}r_t.
$$
Recall that for $g(t)=u^{\frac{4}{n-2}}g$ for some fixed background metric $g$ and $p=\frac{n+2}{n-2}$, we have
$$
R_\psi(t)=u^{-p}(-c_n\Delta_{g_0} u+\psi(x)u).
$$

Let $g(t)=u(t)^{\frac{4}{n-2}}g$. The Yamabe-type flow \eqref{yamabe} may be written as
\begin{equation}\label{yamabe2}
\partial_t u^p=c_n\Delta_{g_0} u-\psi(x)u+r_\psi(t) u^p=-R_\psi(t)u^{p}+r_\psi(t) u^p.
\end{equation}
  We may denote by $\Delta=\Delta_{g_0}$ for simplicity.
Then
$$
\partial_t u^{p+1}=-(p+1)\breve{R}u^{p+1}.
$$
We now compute
$$
\partial_t [\breve{R}^2 u^{p+1}]=[\breve{R}L^u(\breve{R})+pR\breve{R}^2-\breve{R}r_t]u^{p+1}-(p+1)\breve{R}^3u^{p+1}.
$$
Using $\int \breve{R}u^{p+1}=1$, we have
$$
\partial_t \int_M[ \breve{R}^2 u^{p+1}]=\int_M[c_n\breve{R}u\Delta(\breve{R}u)-\psi (\breve{R}u)^2]-\breve{R}^3]u^{p+1}+pr\breve{R}^2u^{p+1},
$$
which is
$$
=-\int_M[c_n|\nabla (\breve{R}u)|^2+\psi (\breve{R}u)^2]-\int_M\breve{R}^3u^{p+1}+pr\int_M\breve{R}^2u^{p+1}.
$$
Since $\psi\geq 0$,
$$
\partial_t \int_M[ \breve{R}^2 u^{p+1}]\leq -\int_M\breve{R}^3u^{p+1}+pr\int_M\breve{R}^2u^{p+1}.
$$
Let
$$
f(t)=\int_M[ \breve{R}^2 u^{p+1}]:=\int_M[ \breve{R}^2 u^{p+1}]dv_{g_0}.
$$
By
$$
f(t)=\int_M[ \breve{R}^2 u^{p+1}]\leq (\int_M[ \breve{R}^3 u^{p+1}])^{2/3},
$$
we know that
$$
f'(t)\leq -f(t)^{3/2}+prf(t).
$$
Recall that
$$
\frac{d}{dt}r(t)=-\frac{n-2}{2}f(t).
$$
Then for any $t>0$,
$$
\int_t^\infty f(s)ds<\infty.
$$
Hence,
$$
\underline{\lim}_{t\to \infty} f(t)=0.
$$
By this and the differential inequality of $f$, we obtain that as $t\to\infty$,
$f(t)\to 0$.
Define
$$
r_\infty=\lim_{t\to \infty} r(t).
$$
Then we have
$$
\lim_{t\to \infty}\int_M|R-r_\infty|^2u^{p+1} dv_{g_0}=0.
$$

We now let $t_k\to \infty$, $u_k=u(t_k)$, and $g_k=g(t_k)=u_k^{4/(n-2)}g_0$. Then
\begin{equation}\label{ps1}
\ \ \int_M dv_{g_k}==\int_M u_k^{2n/(n-2)} dv_{g_0}=1,
\end{equation}
and
$$
\int_M |R_{g_k}-r_\infty|^{2n/(n+2)} dv_{g_k}\to 0,
$$
that is, as $t_k\to \infty$,
\begin{equation}\label{ps2}
\int_M|c_n\Delta u_k-\psi(x)u_k+r_\infty u_k^{\frac{n+2}{n-2}}|^{2n/(n+2)} dv_{g_0}\to 0.
\end{equation}
 We then apply Struwe's compactness result \cite{St} (see Theorem 3.1 in \cite{Hebey} for the detailed proof) to conclude the following result.

\begin{Pro}\label{struwe}  Let $u_k$ be as above with \eqref{ps1} and \eqref{ps2}. After passing to a subsequence, we may find a non-negative integer $m$, a non-negative smooth function $u_\infty$ and a sequence of $m-$tuplets $(x_{j,k}, \epsilon_{j,k})$ with the following properties

(i) The limiting function $u_\infty$ satisfies
$$
c_n\Delta u_\infty-\psi(x)u_\infty+r_\infty u_\infty^{\frac{n+2}{n-2}}=0.
$$

(ii) For $i\not=j$, we have, as $k\to\infty$,
$$
\frac{\epsilon_{j,k}}{\epsilon_{i,k}}+\frac{\epsilon_{i,k}}{\epsilon_{j,k}}+\frac{d(x_{i,k},x_{j,k})^2}{\epsilon_{i,k}\epsilon_{i,k}}\to\infty
$$

(iii) We have as $t_k\to\infty$,
$$
\| u_k-u_\infty- \sum_{i=1}^m w_{(x_{i,k},\epsilon_{i,k})} \|_{H^1(M)}\to 0,
$$
where
$$
w_{(x_{i,k},\epsilon_{i,k})}(z)=\eta_{x_{i,k}} (z)(\frac{4n(n-1)}{r_\infty})^{n-2/4}\epsilon_{i,k}^{n-2/2}(\epsilon_{i,k}^2+d(x_{i,k},z)^2)^{-\frac{n-2}{2}}
$$
with $\eta_{x_{i,k}}(z)$ is the cut-off function defined inside of the ball of the radius $2\delta$ smaller than the injectivity radius on $M$, namely,
$$
\eta_{x_{i,k}}(z)=\eta_\delta(exp_{x_{i,k}}^{-1}(z)), \ \ \eta_\delta\in C^2_0(B_0(2\delta)),
$$
where $B_0(2\delta)\subset \mathbb{R}^n$ the ball with center $0$ and with radius $2\delta>0$.
\end{Pro}
As the consequence of Proposition \ref{struwe}, we know that
\begin{equation}\label{star2}
r_\infty=(E (u_\infty)^{n/2}+mY(S^n)^{n/2})^{2/n}.
\end{equation}

Once we have this result, we may easily prove Theorem \ref{main}.
\begin{proof}
First, we remark that $u_\infty$ is a smooth solution so that if it has a zero point in $M$, by the maximum principle we know that it is identically zero.

Second, by \eqref{star2}, along the flow \eqref{yamabe} we have
$$
Y_\psi(M)\leq r_\infty\leq r(t)<r(0)\leq Y(S^n).
$$

Third, this then implies that $m=0$. Then we have $u_\infty>0$ on $M$, which is the limit of the flow with initial data $u_0$.
Thus we have proved Theorem \ref{main}.
\end{proof}

\section{Convergence part of Yamabe-type flows for negative or flat cases}\label{sect4}

In this section we prove Theorem \ref{thm2}. In the proof below, we may let $p>1$ be any number and we assume $\lambda_1(\psi)<0$.

We now prove Theorem \ref{thm2}.
\begin{proof}

Define the Yamabe-type quotient
$$
Q(g(t))=\frac{\int_M R_\psi(t) dv_t}{(\int_M u(t)^{p+1}dv)^{2/(p+1)}},
$$
where $g(t)=u(t)^{p-1}g$ and we may let
$$
dv_t=u(t)^{p+1}dv.
$$

Along the Yamabe-type flow, we may assume that
$$
\int_Mdv_t=1.
$$
In this case we have
$$
Q(g(t))=r_\psi(t)=\int_M(c_n|\nabla u|^2+\psi(x)u^2)dv/ \int_Mu^{p+1}dv,
$$
i.e.,
$$
Q(g(t))=r_\psi(t)=\int_M(c_n|\nabla u|^2+\psi(x)u^2)dv\geq \lambda_1(\psi)\int u^2dv.
$$
Then by $\int u^2\leq 1$,  we know that
$$
r_\psi(t)\geq \lambda_1(\psi)
$$
for all $t>0$.

Recall that  $r_\psi(t)$ is decreasing in $t$.
Hence, as expected, in case $\lambda_1(\psi)\leq 0$, it is not difficult to see that the factor $u(t)$ is uniformly bounded above and below and as $t\to\infty$, the flow converges to a metric of constant pseudo-scalar curvature. This will be done below.

Recall that $\psi(x)<0$ on $M$. Then
$$\lambda_1(\psi)\leq r_\psi(t)\leq r_\psi(0).$$
Let $u_m(t)=\inf_Mu(t)$. Then by (\ref{yamabe2}) we have
\begin{equation}\label{minimum}
\partial_t u_m^p\geq -\sup_M\psi(x)u_m+r_\psi(t) u_m^p.
\end{equation}
Then
$$
\partial_t u_m^p\geq- \sup_M\psi(x)u_m+ \lambda_1(\psi) u_m^p.
$$
This can be considered in two cases.

\textbf{Case 1}. If $-\sup_M\psi(x)u_m+ \lambda_1(\psi) u_m^p\geq 0$,
then
$$
 u_m^p(t)\geq  u_m^p(0).
$$

\textbf{Case 2}. If $-\sup_M\psi(x)u_m+ \lambda_1(\psi) u_m^p\leq 0$,
then we have
$$
-\sup_M\psi(x)u_m\leq -\lambda_1(\psi) u_m^p.
$$
and
$$
u_m^{p-1}(t)\geq \sup_M\psi(x)/\lambda_1(\psi).
$$
This implies that $u_m(t)$ is uniformly bounded away from zero that
\begin{equation}\label{minimum+1}
u_m^{p-1}(t)\geq \min\{\sup_M\psi(x)/\lambda_1(\psi), u_m^{p-1}(0)\}.
\end{equation}

Similarly for  $u_M(t)=\sup_Mu(t)$, we have
\begin{equation}\label{maximum}
\partial_t u_M^p\leq -\inf_M\psi(x)u_M+ r_\psi(t) u_M^p,
\end{equation}
and by $r_\psi(t)\leq r_\psi(0)$,
$$
\partial_t u_M^p\leq -\inf_M\psi(x)u_M+ r_\psi(0) u_M^p.
$$
We may write this differential inequality as
$$
\frac{p}{p-1}\partial_t u_M^{p-1}\leq -\inf_M\psi(x)+ r_\psi(0) u_M^{p-1}.
$$
When $-\inf_M\psi(x)+ r_\psi(0) u_M^{p-1}\leq 0$, we have
$$
\partial_t u_M^{p-1}\leq 0.
$$
When $-\inf_M\psi(x)+ r_\psi(0) u_M^{p-1}\geq 0$, we have
$$
u_M^{p-1}\leq \inf_M\psi(x)/r_\psi(0).
$$
Then we have
$$
u_M^{p-1}\leq max(\inf_M\psi(x)/r_\psi(0),u_M^{p-1}(0)).
$$
 Then we can get the convergent of the flow $g(t)$ at $t=\infty$.

We claim that $r_\psi(t)$ will eventually become negative, even if this may not be the case at the initial time.
Assume that $r_\psi(t)\geq 0$ for all time, then by (\ref{minimum}) we have
$$
\partial_t u_m^p\geq -\sup_M\psi(x)u_m> 0
$$
and this implies that as $t\to\infty$,
$$
 u_m^p(t)\to \infty,
$$
which is impossible since the volume of $g(t)$ is fixed. Hence, we may choose $t_0>0$ such that
$$
r_\psi(t_0)< 0
$$
and then
$$
r_\psi(t)\leq r_\psi(t_0)< 0
$$
for all $t\geq t_0$. By (\ref{maximum}) we have for $t\geq t_0$,
$$
 u_M^p(t)\leq \max\{u_M^p(t_0), -\frac{1}{r_\psi(t_0)}\sup \psi(x)\}.
$$
This together with (\ref{minimum+1}) implies that $u$ is uniformly bounded from above and away from zero.
Then as in \cite{Y}, we can show that $g(t)$ converges smoothly at an exponential rate to a limit metric with negative pseudo-scalar curvature.

\end{proof}

For comparison, we give an outline of the proof of global existence and convergence result, Theorem \ref{thm3},
in case for any $p>1$ with $\lambda_1(\psi)= 0$ and $\psi(x)=0$ on $M$.

\begin{proof}
The global existence of the flow can be done as in the proof of Theorem \ref{thm1}. So we need only consider the convergence part.
Note that in this case, we have  $r_\psi(t)\geq 0$.
If $r_\psi(0)=0$, then we have $r_\psi(t)= 0$ for all $t>0$ and $r_\psi(t)=0$ for all time. Thus we may assume that
$r_\psi(0)>0$.

By (\ref{minimum}) we have
\begin{equation}\label{rate}
\frac{u_m^p(t)}{u_m^p(0)}\geq \exp(\int_0^tr_\psi(t)).
\end{equation}
We need to get an uniform upper bound for
$u_M(t)=\sup_Mu(t)$.  Recall that
\begin{equation}\label{maximum1}
\partial_t u_M^p\leq  r_\psi(t) u_M^p.
\end{equation}
We use the Gronwall inequality to get
$$
u_M^p(t)/u_M^p(0)\leq \exp(\int_0^tr_\psi(s)ds).
$$
Then
we have
$$
u_M^p(t)/u_M^p(0)\leq \exp(\int_0^tr_\psi(s)ds)\leq \frac{u_m^p(t)}{u_m^p(0)}.
$$
This is the global Harnack inequality.
Therefore, using $\int_M u(t)^{p+1}dv=1$, we have the uniform smooth estimations for $u$. Using \eqref{rate} we know that
$r_\psi(t)\to 0$ as $t\to \infty$.

Multiplying (\ref{yamabe2}) by $u^{1-p}\Delta u$ and integrating, we have
\begin{equation}\label{time}
p \frac{d}{d t} \int_{M}c_n\left|\nabla u\right|^{2} d v+2 \int c_n\left|\Delta u\right|^{2} u^{1-p} d v =2  r_\psi(t) \int_{M}c_n\left|\nabla u\right|^{2} d v
\end{equation}

Using the inequality
\begin{equation}
\int_{M}c_n\left|\Delta u\right|^{2} d v \geq c \int_{M}\left|\nabla u\right|^{2} d v
\end{equation}
for some uniform constant $c>0$ and the fact that $r_\psi(t)\to 0$ as $t\to\infty$ we have
\begin{equation}\label{grad}
\int_{M}\left|\nabla u\right|^{2} d v\leq C e^{-c t}
\end{equation}
Integrating (\ref{time}) in time variable we have for any $T>0$,

\begin{equation}
\int_{T}^{\infty} \int_{M}\left|\Delta u\right|^{2} d v \leq C e^{-c T}
\end{equation}
which says that
\begin{equation}
\int_{T}^{\infty} \int_{M} R_\psi^{2} d v \leq C e^{-c T}
\end{equation}
Then for $t\in [T, T+1)$,
$$
\int_{M} R_\psi^{2} d v \leq C e^{-c t}.
$$
Since $r_\psi(t)$ decreases we have
\begin{equation}
r_\psi(t) \leq C e^{-c t}
\end{equation}
for all $t>0$. By (\ref{grad}) and Poincare inequality we know that $u^p$ converges to its average exponentially in the $L^2$ norm. It follows that $u(t)$ converges exponentially to some limit positive constant.
\end{proof}

\end{document}